\pdfoutput=1

\documentclass[reqno]{amsart}
\usepackage[foot]{amsaddr}
\usepackage{amssymb}
\usepackage{eucal}
\usepackage[hyperindex,breaklinks]{hyperref}
\usepackage{bm}
\usepackage{dsfont}
\usepackage{enumitem}
\setlist[enumerate]{label = \textup{(\alph*)}}

\usepackage[T2A,OT1]{fontenc}
\usepackage{csquotes}
\usepackage[russian,english]{babel}
\usepackage[style = ext-numeric, autolang=other, sorting=none]{biblatex}
\makeatletter
\newcommand{\subalign}[1]{%
  \vcenter{%
    \Let@ \restore@math@cr \default@tag
    \baselineskip\fontdimen10 \scriptfont\tw@
    \advance\baselineskip\fontdimen12 \scriptfont\tw@
    \lineskip\thr@@\fontdimen8 \scriptfont\thr@@
    \lineskiplimit\lineskip
    \ialign{\hfil$\m@th\scriptstyle##$&$\m@th\scriptstyle{}##$\hfil\crcr
      #1\crcr
    }%
  }%
}
\makeatother

\theoremstyle{definition}\newtheorem{Def}{Definition}
\theoremstyle{plain}\newtheorem{Th}{Theorem}
\theoremstyle{plain}

\theoremstyle{remark}
\theoremstyle{plain}\newtheorem{Le}{Lemma}
\theoremstyle{plain}
\theoremstyle{plain}\newtheorem{Cor}{Corollary}
\theoremstyle{plain}

\newcommand{\Cii}[1]{_{{}_{\scriptstyle #1}}}
\newcommand{\cii}[1]{_{{}_{#1}}}
\newcommand{\Av}[2]{\langle {#1}\rangle_{{}_{\scriptstyle #2}}}

\newcommand{\E}{\mathbb{E}}
\newcommand{\cnde}[2]{\E\,\bigl[{#1}\bigm|{#2}\bigr]}
\newcommand{\df}{\buildrel{}_\mathrm{def}\over=}
\newcommand{\chr}{\mathds{1}}
\newcommand{\Rpls}{\mathbb{R}_{\scriptscriptstyle\ge 0}}
\newcommand{\Zpls}{\mathbb{Z}_{\scriptscriptstyle\ge 0}}

\newcommand{\atoms}{\mathcal{A}}

\newcommand{\bell}{\bm B}

\DeclareMathOperator{\osc}{osc}
\DeclareMathOperator{\supp}{supp}

\DeclareMathOperator{\spn}{span}

\DeclareMathOperator{\diam}{diam}

\hypersetup{pdfstartview = FitH, colorlinks = true, linkcolor = blue, citecolor = red}

\title[Bellman function for operators on martingales]{Bellman function method for general operators on martingales: arbitrary regular filtrations}

\author[Nikolay N. Osipov]{Nikolay N. Osipov$^{\ast,\dagger}$}

\address{$^\ast$St. Petersburg Department of V.~A.~Steklov Institute of Mathematics of the Russian Academy of Sciences, St. Petersburg, Russia}
\address{$^\dagger$St. Petersburg State University, St. Petersburg, Russia}

\email{nicknick AT pdmi DOT ras DOT ru}

\thanks{The work was supported by the Russian Science Foundation grant 19-11-00058P}


\begin{document}
\begin{abstract}
    It has been recently shown that the Bellman function method can be applied in the general context of Gundy's extrapolation theorem for vector-valued martingales. But the additional assumption has been made that martingales are adapted to a certain special filtration. Here it is shown that those results can be extended to any regular filtration.
\end{abstract}
\maketitle
\setcounter{secnumdepth}{1}
\setcounter{tocdepth}{1}

\section{Introduction} 
Gundy's extrapolation theorem~\cite{Gu1968} and especially its version~\cite{Ki1985tran} for vector-valued martingales can be considered as a martingale counterpart of the fact that very general Calder\'on--Zygmund operators are bounded. For example, it leads to a discrete analog~\cite{Os2016tran} (for the Walsh basis) of Rubio de Francia's 
inequality, while in the original Fourier setting~\cite{Ru1985} the proof requires the boundedness of a Cal\-de\-r\'on--Zyg\-mund operator with a very weak and subtle smoothness condition.

On the other hand, the most basic example of Gundy's operators is the martingale transforms, which can be considered as a martingale counterpart of the Hilbert transform. In~\cite{Bu1984} Burkholder proves the $L^p$-bo\-un\-ded\-ness 
only for the martingale transforms, but with the method that gives a deep insight into the structure of the estimated $L^p$-norms and, in particular, results in sharp constants in the corresponding $L^p$-inequalities. 
His methodology gives rise to a new theory \cite{NaTr1996tran, Ose2012, SIZOV2015,VaVo2020, SIZV2023} and is now commonly referred to as the Bellman function method in harmonic analysis.

In~\cite{BoOsTs2022} the authors demonstrate that Burkholder's approach can be applied in the general context of Gundy's theorem for vector-valued martingales, but under the additional assumption that martingales are adapted to a certain special filtration. Here we show that those results can be extended to any regular filtration, which makes the approach outlined in~\cite{BoOsTs2022} truly general.

\section{Preliminaries}
Let $\mathcal{F}_0 \subseteq \mathcal{F}_1\subseteq \dots$ be a filtration of the Borel $\sigma$-algebra 
on an interval~$I\subset\mathbb{R}$. We also set 
$\mathcal{F}_\infty \df \sigma(\cup_n\mathcal{F}_n)$. 
We take $\delta \in \bigl(0, \frac{1}{2}\bigr]$ and
impose the following regularity condition on the filtration.
\begin{enumerate}[label = \textup{(R)}]
    \item\label{it:R} We have $\mathcal{F}_0 = \{I,\emptyset\}$, and
    each algebra $\mathcal{F}_n$ is finite. 
    If 
    $J$ is an atom of $\mathcal{F}_n$, 
    $Q$ is an atom of $\mathcal{F}_{n+1}$, 
    and $Q\subseteq J$, then $|Q|/|J|\ge\delta>0$.
\end{enumerate} 
Let $\mathcal{H}$ be a separable Hilbert space. 
For $f \in L^1(I,\mathcal{H})$, we denote 
$$\E_n f\df \cnde{f}{\mathcal{F}_n}. 
$$
A~sequence~$\{f_n\}$ of functions in $L^1(I,\mathcal{H})$ is called a martingale 
if 
$
\E_m f_n = f_{m}
$
for~$m\le n$.

We say that a martingale $\{f_n\}$ is $L^p$-bounded if 
$$
\|\{f_n\}\|\Cii{L^p} \df \sup_n \|f_n\|\Cii{L^p} < \infty.
$$
In this case, if $1<p<\infty$, then $f_n$ converge in $L^p(I,\mathcal{F}_\infty,\mathcal{H})$ to a function~$f$ such that 
$f_n = \E_n f$ and $\|\{f_n\}\|\Cii{L^p} = \|f\|\Cii{L^p}$.
On the other hand, any function $f \in L^p(I,\mathcal{F}_\infty,\mathcal{H})$, $1\le p <\infty$, generates the $L^p$-bounded martingale $\{\E_n f\}$ such that 
$$
\E_n f \xrightarrow{L^p} f\quad\mbox{and}\quad   \|\{\E_n f\}\|\Cii{L^p} =\|f\|\Cii{L^p}.
$$
Thus, for $1< p <\infty$ we identify $L^p$-bounded martingales with their limit functions.
Details can be found in the manuals~\cite{DiUh1977,Wi1991}.

We call a martingale simple if $f_{n+1}\equiv f_n$ for all sufficiently large~$n$. In particular, a simple martingale is $L^p$-bounded for any $1\le p \le \infty$.
Let $S(I,\mathcal{H})$ be the space of simple martingales, and let $\mathcal{L}(I, \mathcal{H})$ be the space of all linear operators that transform simple martingales into measurable scalar functions on~$I$.
\begin{Def}\label{def:class_G}
    Let $T\in \mathcal{L}(I, \mathcal{H})$. 
    We say that $T$ belongs to the class $\mathcal{G}(I,\mathcal{H})$ if it has the following properties.
\begin{enumerate}[label = \textup{(G\arabic*)}]
    \item\label{it:G1} $\|Tf\|\Cii{L^2} \le \|f\|\Cii{L^2}$.
    \item\label{it:G2} If $f$ satisfies the relations $\Delta_0f \df f_0\equiv \bm 0$ and 
    $$
        \Delta_n f \df f_n-f_{n-1} = \chr_{e_n} \Delta_n f
    $$
    for $n>0$ and some $e_n\in\mathcal{F}_{n-1}$, then 
    $$
    \{|Tf| > 0\} \subset \bigcup_{n>0} e_n.
    $$
\end{enumerate}
\end{Def}
Due to property~\ref{it:G1}, we can treat any $T \in \mathcal{G}(I,\mathcal{H})$ as a bounded linear operator from 
$L^2(I,\mathcal{F}_\infty,\mathcal{H})$ to $L^2(I)$.
The version~\cite{Ki1985tran} of Gundy's theorem~\cite{Gu1968} for operators on vector-valued martingales implies that operators $T \in \mathcal{G}(I,\mathcal{H})$ are $L^p$-bounded for $1<p\le 2$. Our goal is to describe how the Bellman function method works for this $L^p$-boundedness.

\section{Results}
Henceforth, we suppose $1<p\le 2$, and $\tfrac{1}{p}+\tfrac{1}{q} = 1$. 
We agree that for vectors $x,y \in \mathcal{H}$, the notation~$xy$ means their inner product, and $|x|$ means the $\mathcal{H}$-norm of~$x$. 
For $f \in L^2(I,\mathcal{H})$ and $J \in \mathcal{F}_\infty$, 
we set
$$\Av{f}{J} \df \frac{1}{|J|}\int_J f$$
and
$$
    \osc_J^2(f) \df \Bigl\langle\bigl|f-\Av{f}{J}\bigr|^2\Bigr\rangle_{J}
    =\bigl\langle|f|^2\bigr\rangle_{J} - \bigl|\Av{f}{J}\bigr|^2.
$$ 

Suppose $T\in\mathcal{G}(I,\mathcal{H})$. 
We have $\|T^*\|\cii{L^2\to L^2} =\|T\|\cii{L^2\to L^2}
\le 1$, and thus the inequality
\begin{equation}\label{eq:x2_ineq}
    \Av{g^2}{I} - \osc_I^2(T^*g) \ge \bigl|\Av{T^*g}{I}\bigr|^2 > 0
\end{equation}
holds for any $g\in L^2(I)$. 

We introduce the Bellman function
\begin{equation}\label{eq:bell_def}
	\bell(x) \df \sup\left\{\bigl\langle g\,T\bigl[f-\Av{f}{I}\bigr]\bigr\rangle\Cii{I}
	\,\middle|\; 
	\begin{aligned}
	&\Av{f}{I} = x_1,\\[1pt] 
	&\Av{g^2}{I} - \osc_I^2(T^*g) = x_2,\\[1pt] 
	&\Av{|f|^p}{I} = x_3,\; \Av{|g|^q}{I} = x_4
	\end{aligned}
	\right\},
\end{equation}
where $x = (x_1,x_2,x_3,x_4)\in \mathcal{H}\times\Rpls^3$ and the supremum is taken over 
$$f\in S(I,\mathcal{H}),\quad g\in L^2(I),\quad\mbox{and}\quad T\in\mathcal{G}(I,\mathcal{H})$$
satisfying the identities after the vertical bar. 

Let $\Omega_{\bell}$ consist of all the points~$x$ for which the supremum in~\eqref{eq:bell_def} is taken over a non-empty set. 
Applying Jensen's 
inequality in vector and scalar forms (or H\"older's inequality together with Minkowski's integral inequality for the $\mathcal{H}$-norm), we obtain
$$
    \Omega_{\bell} \subseteq \Omega_p \df \bigl\{x \in \mathcal{H}\times\Rpls^3\bigm| |x_1|^p \le x_3,\, x_2^q \le x_4^{2} \bigr\}.
$$ 
\begin{Def}\label{def:class_K}
We say that a function $B\in C(\Omega_p)$ belongs to 
the class $\mathcal{K}^p_\delta(\mathcal{H})$ for $\delta \in \bigl(0, \frac{1}{2}\bigr]$
if it satisfies the following boundary condition and geometric concave-type condition.
\begin{enumerate}[label = \textup{(B\arabic*)}]
    \item\label{it:bnd_cnd} If $|x_1|^p = x_3$ then $B(x) \ge 0$.
    \item\label{it:cnc_cnd} If for $N\in\mathbb{N}$, $x, x^1,\dots,x^N \in \Omega_p$, $d \in \mathbb{R}$, and 
    $\lambda_1,\dots,\lambda_N \in \Rpls$ such that $\lambda_1+\dots+\lambda_N=1$ 
    and $\lambda_k \ge \delta$, we have
    \begin{equation*}
        \sum_{k=1}^N\lambda_k x^k - x = (\bm{0}, d^2, 0, 0),
    \end{equation*}
    then
    \begin{equation}\label{eq:main_ineq}
        B(x) \ge |d|\,\diam\bigl\{x_1^1,\dots,x_1^N\bigr\} + \sum_{k=1}^N\lambda_k B(x^k).
    \end{equation}
\end{enumerate}
\end{Def}
\begin{Le}\label{le:any_dlt}
    If $B\in\mathcal{K}^p_{{1}/{2}}(\mathcal{H})$, then for any $\delta \in \bigl(0, \frac{1}{2}\bigr]$ there exists a constant $C_\delta > 0$ such that
    ${C_\delta B \in \mathcal{K}^p_{\delta}(\mathcal{H})}$.
\end{Le}
\begin{Th}\label{th:BleB}
If $\delta \in \bigl(0, \frac{1}{2}\bigr]$ is the parameter from the regularity condition~\ref{it:R} and ${B\in\mathcal{K}_\delta^p(\mathcal{H})}$, then $\bell(x) \le B(x)$ for all $x\in \Omega_{\bell}$.
\end{Th}
Using the function $\mathcal{B} \in \mathcal{K}^p_{{1}/{2}}(\mathcal{H})$
constructed explicitly in~\cite{BoOsTs2022}, we can deduce, from Lemma~\ref{le:any_dlt} and Theorem~\ref{th:BleB}, the following consequence.
\begin{Cor}\label{cor}
    If $T\in\mathcal{G}(I,\mathcal{H})$ and $f\in S(I,\mathcal{H})$, then for $1< p \le 2$ we have
    $$
     \|Tf\|\Cii{L^p} \le C_{p,\delta}\, \|f\|\Cii{L^p},
    $$
    and $T$ can be continuously extended to $L^p(I,\mathcal{F}_\infty,\mathcal{H})$.
\end{Cor}

\section{Proof of Lemma~\ref{le:any_dlt}}
Since $B \in C(\Omega_p)$, it suffices to verify~\ref{it:cnc_cnd} in a situation where~$\lambda_k$ are dyadic rationals. 
Thus, we assume that $\lambda_k = \frac{a_k}{b}$, where $b = a_1+\dots+a_N = 2^M$. We denote 
$$\mathcal{X}\df \bigl\{x_1^1,\dots,x_1^N\bigr\} \subset\mathcal{H}.$$
Let $y_1,y_2 \in \mathcal{X}$ be such that $|y_1-y_2| = \diam\mathcal{X}$. Let $z_1,\dots,z_b$ be a sequence that is obtained by making 
$a_k$ copies of each $x_1^k$ and sorting the result by $|v_i|$, where 
$v_i\df P_{y_2-y_1}(z_i - y_1)$. Here $P_{y_2-y_1}$ is the orthogonal projection onto the line $\spn \{y_2-y_1\}$.
Next, we have $\delta  \le \frac{a_i}{b} \le 1 - \delta(N-1)$ and $N \le \frac{1}{\delta}$. In particular, this implies that
\begin{equation}\label{eq:a_asymp_b}
    a_1 \asymp\dots\asymp a_N \asymp b,
\end{equation}
where $a\asymp b$ means that $c_\delta b \le a \le C_\delta b$. We denote
$$
V_1 \df \frac{v_1 + \dots + v_{b/2}}{b/2}\quad\mbox{and}\quad V_2 \df \frac{v_{b/2+1} + \dots + v_{b}}{b/2}.
$$
Without loss of generality we assume that $z_1 = y_1 = x_1^1$ and $z_b = y_2 = x_1^N$. Then the vector~$v_1$ appears in the formula for $V_1$ at least $\min(b/2,a_1)$ times,
and the vector~$v_b$ appears in the formula for $V_2$ at least $\min(b/2,a_N)$ times. Thus, due to~\eqref{eq:a_asymp_b} we have $|V_1 - v_{b/2}|\asymp |v_1 - v_{b/2}|$ 
and $|V_2 - v_{b/2+1}|\asymp |v_b - v_{b/2+1}|$. This implies that
\begin{equation}\label{eq:to_diam}
\Bigl|\frac{z_1 + \dots + z_{b/2}}{b/2} - \frac{z_{b/2+1} + \dots + z_{b}}{b/2}\Bigr| \ge |V_1 - V_2| \ge c_\delta |v_1 - v_b| 
= c_\delta\diam\mathcal{X}.
\end{equation}
Now, using once the dyadic version of \eqref{eq:main_ineq} together with~\eqref{eq:to_diam}, and after that additionally using $M-1$ times the midpoint concavity of $B$, we conclude that
$\frac{1}{c_\delta}B$ satisfies~\eqref{eq:main_ineq}. \qed

\section{Proof of Theorem~\ref{th:BleB}}
First, we need certain preliminary constructions and lemmas.
By $\atoms_m$ we denote the set of atoms of the corresponding algebra $\mathcal{F}_m$, 
$m \in \Zpls \cup \{\infty\}$. 
We also introduce the set $\mathcal{D} \df (\cup_n\atoms_n) \setminus \atoms_{\infty}$ of counterparts of dyadic intervals: the set of all the atoms that do not remain stable and are eventually split.
From the filtration $\{\mathcal{F}_n\}$, we obtain a new filtration $\{\mathcal{F}_0,\mathcal{F}_{J}\}\cii{J\in\mathcal{D}}$ as follows. 
We leave $\mathcal{F}_0$ as it is. Next, for each $n \in \Zpls$ in ascending order, we consider, in an arbitrary but fixed order, all the atoms
$J \in \atoms_n \setminus \atoms_{n+1}$, i.~e. the atoms that are split at time $n$ (if any). For each such~$J$, we build $\mathcal{F}_J$ as the minimum algebra that contains the new atoms $Q\in \atoms_{n+1}$, $Q \subseteq J$,
together with the previous algebra in the filtration under construction. 
For each algebra~$\mathcal{F}_J$ in the new filtration, we introduce the set~$\atoms_J$ of its atoms. 

For an algebra $\mathcal{F}_J$ (where $J \in \mathcal{D}$ are already split) we consider the previous algebra $\mathcal{F}_J^{\,\mathrm{prev}} \subset \mathcal{F}_J$ in the filtration (where $J$ is still uncut and other atoms are the same). For $f \in L^1(I,\mathcal{H})$, we introduce the operators
$$
    \Delta_0 f \df\E_0 f\quad\mbox{and}\quad 
    \Delta_J f \df \cnde{f}{\mathcal{F}_J} - \cnde{f}{\mathcal{F}_J^{\,\mathrm{prev}}},\quad J \in \mathcal{D}.
$$
We note that $\supp\Delta_J f \subseteq J$. 

\begin{Le}\label{le:orth1}
  The operators $\Delta_J$ are orthogonal projections in $L^2(I,\mathcal{H})$. 
\end{Le}
\begin{proof}
In order to prove that the projections~$\Delta_J$ are orthogonal projections, it is necessary and sufficient to check that they are self-adjoint. Let $f,g\in L^2(I,\mathcal{H})$. We have
\begin{equation}\label{eq:EE}
     \E\,\bigl[f\, \Delta_J g\bigr] = \E\,\bigl[\cnde{f\, \Delta_J g}{\mathcal{F}_J}\bigr]
     = \E\,\bigl[\Delta_J g\,\cnde{f}{\mathcal{F}_J}\bigr]
\end{equation}
and
\begin{equation}\label{eq:zero}
0 = \E\,\bigl[\Delta_J g\,\cnde{f}{\mathcal{F}_J^{\,\mathrm{prev}}}\bigr].
\end{equation}
By subtracting~\eqref{eq:zero} from \eqref{eq:EE} and using symmetry, we have
$$
    \E\,\bigl[f\, \Delta_J g\bigr] = \E\,\bigl[\Delta_J f\, \Delta_J g\bigr] = \E\,\bigl[g\,\Delta_J f\bigr].
$$
\end{proof}
\begin{Le}\label{le:orth2}
  For $J \in \mathcal{D}$, the subspaces $\Delta_J L^2(I, \mathcal{H})$ are mutually orthogonal.
\end{Le}
\begin{proof}
It is obvious that
$$
\E\,\bigl[\Delta_Q f\, \Delta_J g\bigr] = 0
$$
whenever $Q\cap J = \emptyset$ or $Q \subsetneq J$. 
\end{proof}
We also need the following localization property.
\begin{Le}\label{le:loc}
    Suppose $T \in \mathcal{G}(I,\mathcal{H})$ and $f \in \Delta_J L^2(I, \mathcal{H})$ for $J \in \mathcal{D}$. 
    Then we have $\supp Tf \subseteq J$.
\end{Le}
\begin{proof}
By the construction of $\mathcal{F}_J$, there exists $n>0$ such that 
$$f = \Delta_n f = \chr_{J} \Delta_n f \quad\mbox{and}\quad J \in \mathcal{F}_{n-1}.$$ 
Thus, by~\ref{it:G2} we have $\supp Tf \subseteq J$.
\end{proof}

Now we are ready to prove the theorem.
Fix $x\in\Omega_{\bell}$ and consider $f\in S(I,\mathcal{H})$, $g\in L^2(I)$, and $T\in\mathcal{G}(I,\mathcal{H})$ such that $x=x^I$, where
$$
x^J = \bigl(x_1^J,x_2^J,x_3^J,x_4^J\bigr) \df \big(\Av{f}{J},\,\Av{g^2}{J} - \osc_J^2(T^*g),\,\Av{|f|^p}{J},\,\Av{|g|^q}{J}\big),\quad J \in \mathcal{D}. 
$$
Representing the restriction $(T^*g)\big|_{J} \in L^2(I,\mathcal{F}_\infty,\mathcal{H})$ as the sum of the corresponding martingale differences with respect to $\{\mathcal{F}_J\}$ and using their pairwise orthogonality (see Lemma~\ref{le:orth2}), 
we obtain 
\begin{equation}\label{eq:osc_ser}
    \osc_J^2(T^*g) = \frac{1}{|J|}\sum_{Q\subseteq J, Q \in \mathcal{D}} \bigl\|\Delta_Q T^* g\bigr\|_{L^2(\mathcal{H})}^2.
\end{equation}
Now we prove that the property $x_2^I \ge 0$ is inherited by $x_2^J$ for all $J \in \mathcal{D}$.
This is the very place where we need~\ref{it:G2}.
Starting with~\eqref{eq:osc_ser} and applying the fact that the operators~$\Delta_J$ are self-adjoint (see Lemma~\ref{le:orth1}) together with the localization property of~$T$ (see Lemma~\ref{le:loc}), we obtain 
\begin{align*}
\osc_J^2(T^*g) &= \frac{1}{|J|}\sum_{Q\subseteq J, Q \in \mathcal{D}} 
\bigl(g,\,T \Delta_Q T^* g\bigr)_{L^2(\mathcal{H})}\\
&=\frac{1}{|J|}\sum_{Q\subseteq J, Q \in \mathcal{D}} 
\bigl(g|\Cii{J},\,T \Delta_Q T^* g\bigr)_{L^2(\mathcal{H})}\\
&=\frac{1}{|J|}\sum_{Q\subseteq J, Q \in \mathcal{D}} 
\bigl(T \Delta_Q T^*\bigl[g|\Cii{J}\bigr],\, g|\Cii{J}\bigr)_{L^2(\mathcal{H})}\\
&=\frac{1}{|J|}\sum_{Q \in \mathcal{D}} \bigl\|\Delta_Q T^* \bigl[g|\Cii{J}\bigr]\bigr\|_{L^2(\mathcal{H})}^2
=\frac{|I|}{|J|}\osc_I^2\bigl(T^*\bigl[g|\Cii{J}\bigr]\bigr).
\end{align*}
Thus, applying~\eqref{eq:x2_ineq} to $g|\Cii{J}$, we have $x_2^J \ge 0$.

Next, we set
$$
d_{J} \df {|J|^{-1/2}}\bigl\|\Delta_J T^* g\bigr\|_{L^2(\mathcal{H})}.
$$
Due to~\eqref{eq:osc_ser}, we have
\begin{equation}\label{eq:d2_x2}
    d_{J}^2= \sum_{Q\subseteq J, Q\in\atoms_J}\frac{|Q|}{|J|}x_2^{Q} - x_2^J.
\end{equation}
We also obtain
\begin{multline}\label{eq:diam_gTf}
|d_J|\,\diam\bigl\{x_1^{Q}\bigr\}_{Q\subseteq J, Q\in\atoms_J}\ge 
|d_J|\,\max\bigl\{\bigl|x_1^{Q} - x_1^J\bigr|\bigr\}_{Q\subseteq J, Q\in\atoms_J}=
|d_J|\,\max|\Delta_J f| \\[5pt]
\ge \frac{1}{|J|}\bigl\|\Delta_J T^* g\bigr\|_{L^2(\mathcal{H})}
\bigl\|\Delta_J f\bigr\|_{L^2(\mathcal{H})} 
\ge \frac{1}{|J|} \bigl(\Delta_J T^* g, \Delta_J f\bigr)_{L^2(\mathcal{H})}.
\end{multline}
Relying on~\eqref{eq:d2_x2} and~\eqref{eq:diam_gTf} and recursively applying inequality~\eqref{eq:main_ineq} up to a level~$n$, we obtain
\begin{equation}\label{eq:bell_ind}
    B(x) \ge \frac{1}{|I|}\sum_{J\in \!\!\bigcup\limits_{j<n}\!\!\!\atoms_{j}\setminus\atoms_n} \bigl(\Delta_J T^* g, \Delta_J f\bigr)_{L^2(\mathcal{H})}
    +\frac{1}{|I|}\sum_{J \in \atoms_n}|J| B(x^J).
\end{equation}
We introduce the step function
\[
x^n(t) = (x_1^n(t),x_2^n(t),x_3^n(t),x_4^n(t))
\]
that takes values $x^J$ on the intervals $J \in \atoms_n$.
Since $f$ is a simple martingale, there exists $n$ such that inequality~\eqref{eq:bell_ind} can be written as
$$
    B(x) \ge \frac{1}{|I|}\sum_{J\in\mathcal{D}} \bigl(\Delta_J T^* g, \Delta_J f\bigr)_{L^2(\mathcal{H})}
    +\int\limits_I B\bigl(f(t), x_2^n(t), |f(t)|^p, x_4^n(t)\bigr)\,dt.
$$
Using Lemma~\ref{le:orth2} and the boundary condition~\ref{it:bnd_cnd}, we arrive at the desired estimate
$$
    B(x) \ge \bigl\langle g\,T\bigl[f-\Av{f}{I}\bigr]\bigr\rangle\Cii{I}.\eqno\qed
$$


\section{Proof of Corollary~\ref{cor}}
Due to~\cite{BoOsTs2022}, we have a function $\mathcal{B} \in \mathcal{K}^p_{{1}/{2}}(\mathcal{H})$ that has a form
\begin{equation}\label{eq:B}
\mathcal{B}(x) = C_p (x_3 + x_4) - h(x_1, x_2),
\end{equation}
where $h$ is a non-negative function.
By Lemma~\ref{le:any_dlt}, we have a constant~$C_\delta > 0$ such that
\begin{equation}\label{eq:B_dlt}
\mathcal{B}_\delta \df C_\delta \mathcal{B} \in \mathcal{K}^p_{\delta}(\mathcal{H}).
\end{equation}

Now we proceed in much the same way as in~\cite{BoOsTs2022}. Suppose $f \in S(I,\mathcal{H})$, ${g\in L^q(I)}$, and
$$
    x\df\Big(\Av{f}{I},\,\tfrac{1}{|I|}\|g\|\Cii{L^2}^2 - \osc_{I}^2(T^*g),\, 
\tfrac{1}{|I|}\|f\|\Cii{L^p}^p,\,\tfrac{1}{|I|}\|g\|\Cii{L^q}^q\Big).
$$
Let $\lambda > 0$. By the homogeneity of~$\bell$ and by Theorem~\ref{th:BleB} combined 
with~\eqref{eq:B} and~\eqref{eq:B_dlt}, we obtain
\begin{equation}\label{eq:cor_hom}
\begin{aligned}
    \Av{g\,Tf}{I} - \Av{f}{I}\,\Av{T^*g}{I} &\le \bell(x_1,x_2,x_3,x_4)\\
    &= \bell\bigl(\lambda x_1,\,\lambda^{-2}x_2,\,\lambda^p x_3,\,\lambda^{-q}x_4\bigr)\\
    &\le \mathcal{B}_\delta\bigl(\lambda x_1,\,\lambda^{-2}x_2,\,\lambda^p x_3,\,\lambda^{-q}x_4\bigr)\\
    &\le C_\delta\, C_p\,(\lambda^p x_3+\lambda^{-q}x_4).
\end{aligned}
\end{equation}

In order to guess optimal~$\lambda$, we need to solve the equation 
$$\partial_{\lambda}\big[\lambda^p x_3 + \lambda^{-q}x_4\big] = 0.$$
We obtain 
\begin{equation}\label{eq:cor_lmd}
\lambda = \biggl(\frac{q\,x_4}{p\,x_3}\biggr)^{\tfrac{1}{p+q}}.
\end{equation}

By H\"older's inequality, we get
$$
\int_I |T^*g| \le |I|^{1/2} \|T^*g\|\Cii{L^2} \le |I|^{1/2} \|g\|\Cii{L^2} \le |I|^{1-1/q} \|g\|\Cii{L^q}
$$
and
$$
\int_I |f| \le |I|^{1/q} \|f\|\Cii{L^p}.
$$
Thus, we have
\begin{equation}\label{eq:cor_Hld}
    \bigl|\Av{f}{I}\Av{T^*g}{I}\bigr| 
    \le \frac{1}{|I|}\|f\|\Cii{L^p}\|g\|\Cii{L^q}.
\end{equation}
Combining~\eqref{eq:cor_hom}, \eqref{eq:cor_lmd}, and~\eqref{eq:cor_Hld} for $g$ and $-g$, we 
obtain
$$
    \biggl|\int_I g\,Tf\biggr| \le C_{p,\delta}\,\|f\|\Cii{L^p}\|g\|\Cii{L^q}.
$$
Thus, we have the desired inequality for $f \in S(I,\mathcal{H})$. 
Finally, any function $f$ in $L^p(I,\mathcal{F}_\infty,\mathcal{H})$ can be approximated by its expectations $\E_n f$, and we are done.\qed

\printbibliography
\end{document}